\documentclass{amsart}
\usepackage[all]{xy}
\usepackage{ulem}
\usepackage{comment}
\usepackage{xcolor}
\usepackage{soul}
\usepackage{amsmath}
\usepackage{stmaryrd}
\usepackage{amsfonts}
\usepackage{amssymb}
\usepackage{amsthm}
\usepackage{amscd}
\usepackage{latexsym}
\usepackage{txfonts}
\usepackage{float}
\usepackage{enumerate}
\usepackage{tabularx}
\usepackage{hyperref}
\usepackage{comment}
\theoremstyle{plain}
\newtheorem{theorem}{Theorem}[section]
\newtheorem{lemma}[theorem]{Lemma}
\newtheorem{cor}[theorem]{Corollary}

\newtheorem*{theorem*}{Theorem}
\theoremstyle{remark}
\theoremstyle{remark} \newtheorem*{remark}{Remark}
\theoremstyle{remark} 
\newcommand{\bb}{\mathbb}
\newcommand{\mcal}{\mathcal}
\newcommand{\mrm}{\mathrm}
\newcommand{\mbf}{\mathbf}
\newcommand{\mfk}{\mathfrak}
\newcommand{\ov}{\overline}
\newcommand{\irr}{\mathrm{Irr}}
\newcommand{\cS}{\mathcal{S}}
\newcommand{\cC}{\mathcal{C}}
\newcommand{\cA}{\mathcal{A}}
\newcommand{\BM}{\mathrm{span}\ \mathcal{A}}
\newcommand{\ignore}[1]{{}}
\newcommand{\abs}[1]{|#1|}

\title{Perfect and Multiple state Transfer in Oriented Cayley Graphs}

\author{Ada Chan}
\address{Department of Mathematics and Statistics, York University, Toronto, ON, Canada N3J 1P3}
\email{ssachan@yorku.ca}

\author[V.R.T. Pantangi]{Venkata Raghu Tej Pantangi}
\address{}
\email{pvrt1990@gmail.com}
\author[S.Razafimahatratra]{Andriaherimanana Sarobidy Razafimahatratra}
\address{School of Mathematics and Statistics, Carleton University, 1125 Colonel by drive, Ottawa, ON K1S 5B6, Canada}
\email{sarobidy@phystech.edu}
\author{Peter Sin}
\address{Department of Mathematics, University of Florida, P. O. Box 118105, Gainesville FL 32611, USA}
\email{sin@ufl.edu}

\thanks{This work was partially supported by a grant from the Simons Foundation (\#633214
 to Peter Sin). A. S. Razafimahatratra was supported by the Fields-AIMS-Perimeter Postdoctoral Fellowship.
 A. Chan gratefully acknowledges the support of the NSERC Grant No. RGPIN-2021-03609.}

\begin{document}
\begin{abstract} We study perfect state transfer and multiple state transfer in oriented normal Cayley graphs. We construct examples in a variety of groups, ranging from abelian to nonsolvable, and establish some general restrictions and nonexistence results.
\end{abstract}
\maketitle
\section{Introduction}
An oriented graph is a directed graph with no loops or multiple edges. Given a directed graph $X$, we denote by $\mrm{V}(X)$  the vertex set, and by $\mrm{A}(X) \subset \mrm{V}(X)\times \mrm{V}(X)$ the set of arcs. The adjacency matrix of an oriented graph $X$ is the matrix $A_X$ whose rows and columns are indexed by $\mrm{V}(X)$ (in some order) satisfying 
\[ 
A_{a,b} = \begin{cases} 1 & \text{if $(a,b) \in \mrm{A}(X)$,} \\
-1 & \text{if $(b,a) \in \mrm{A}(X)$,} \\
0 & \text{otherwise.}
\end{cases}
\] 

Zimbor\'as et al. \cite{Zimboras_et_al} and Chavas et al. \cite[Theorem 2.1]{Chaves_et_al} have shown that the Hermitian matrix
$iA_X$ can be realized as the restriction to the $1$-excitation subspace of a suitable Hamiltonian
for a system of $n$ qubits. The time evolution of some initial state in this subspace is what
is meant by a quantum walk on an oriented graph $X$, and is found by solving the
Schr\"odinger equation. Thus, a quantum walk on an oriented graph is determined by an initial state and the time-dependent transition matrices, $U_X(t)=e^{-it(iA_X)}=e^{tA_X}$. 
This extension of quantum walks on undirected graphs to oriented graphs was first defined in \cite{cameron2014universal}. Of central importance to the theory is the phenomenon of {\it perfect state transfer (PST)}. Given a vertex $b\in V(X)$ and a matrix $B$ indexed by $V(X)$ in its rows and columns, we shall denote by $Bb$ the column of $B$ corresponding to $b$, which is equivalent to identifying $b \in V(X)$ with the corresponding canonical vector of $b$ in $\mathbb{C}^{V(X)}$. We say that perfect state transfer occurs between a vertex $a \in \mrm{V}(X)$ to a vertex $b \in \mrm{V}(X)$ at a time $\tau$ if $U_X(\tau)a = \lambda b$, for some complex number $\lambda $ of norm $1$. 
If $a=b$, then we say $a$ is {\it periodic}.  As $U_X(\tau)$ is a real matrix, we must have $\lambda= \pm 1$. The quantity $\lambda$ is known as the phase.  (We drop the subscripts in $A_X$ and $U_X(t)$ if the graph is clear from the context.)

For $a\in V(X)$, let
\begin{equation*}
    \cS_a = \left\{ b \in V(X) \ \vert \ U(\tau)a = \lambda b, \text{for some $\tau > 0$ and $\vert\lambda\vert=1$}\right\}.
\end{equation*}
Suppose PST occurs from $a$ to $b\ (\neq a)$ in $X$.
It follows from Theorem~8.1 and Lemma~8.3 of \cite{GodsilRST} that $a$ is periodic and $a\in \cS_a$.  Let $b$ and $c$ be distinct vertices in $\cS_a$.  For $u\in \{a,b,c\}$, let $\tau_u>0$ denote the minimum time where $U(\tau_u)a = \lambda_u u$, for some phase $\lambda_u$. 
Note that $\tau_b, \tau_c \leq \tau_a$.
If $\tau_b < \tau_c$ then PST occurs from $b$ to $c$ at time $(\tau_c-\tau_b)$,
and PST occurs from $c$ to $b$ at time $(\tau_a-\tau_c+\tau_b)$.  
Hence PST occurs between every pair of distinct vertices in $\cS_a$, and $\cS_b=\cS_a$ for $b\in \cS_a$.  
When $\vert \cS_a\vert \geq 3$, $X$ is said to have {\it multiple state transfer} (MST) on $\cS_a$. This phenomenon cannot be observed in undirected graphs as PST in such graphs is monogamous (c.f \cite{kay}). However, \cite{cameron2014universal, godsil2020perfect, acuaviva2023state, godsil2024oriented} provide examples of multiple state transfer in oriented graphs. 

The purpose of the present paper is to study state transfer in oriented normal Cayley graphs. The spectral
theory of such a graph is closely related to
the character theory of the underlying group,
which provides powerful machinery for our investigation.
We have two main goals. The first is  to construct examples of perfect and multiple state transfer for as wide a variety of groups as possible, ranging from  abelian to nonsolvable groups. The second is to find theoretical restrictions on the existence of state transfer on these graphs.

The paper is organized as follows. 
We start in \S~\ref{sec:BM} with some background on oriented graphs in Bose-Mesner algebras, of which
normal oriented Cayley graphs are a particular case.
A key result is that, for any vertex $a$ involved in PST, $\cS_a$ is permuted cyclically by a permutation matrix in the Bose-Mesner algebra.  Moreover, 
we show that $\abs{\cS_a}$ must lie in the set $\{2,3,4,6\}$.

Next, in \S~\ref{sec:Cayley} we develop the theory for oriented normal Cayley graphs.  The cyclic permutation of $\cS_a$ of \S~\ref{sec:BM} takes the form of a
central element of order $\abs{\cS_a}$. We obtain a characterization of PST in terms of the characters of the group and its values on the connection set and this central element. This characterization is used later for the construction of examples of PST/MST, but also provides an easily-checked criterion for nonexistence of PST in any oriented normal Cayley graph on a given group based on the field of character values.

In \S~\ref{sec:solvable} we show that for solvable groups the case $\abs{\cS_a}=6$ can be ruled out. 
The proof is an application of the nonexistence criterion, together with the theory of $p$-special characters of solvable groups from \cite{Gajendragadkar}.
The general question of existence of an oriented graph 
with MST on a set of size 6 remains open.
Our results show that, in order for such a graph to be an oriented normal Cayley graph, the underlying group must be nonsolvable.

The remaining sections, which describe examples of PST in various classes of groups, may be read independently of each other.

In \S~\ref{sec:abelian}, we consider PST in oriented Cayley graphs on ${\mathbb Z_3}^n$ and ${\mathbb Z_4}^n$, constructing examples of PST and
ruling out MST on sets of size 4.

Examples of MST on sets of size 3 in oriented normal Cayley graphs on nonabelian groups are constructed on
extraspecial $3$-groups in \S~\ref{sec:extraspecial} 
and some nonsolvable examples are described in \S~\ref{sec:nonsolvable}.

Examples of MST on sets of size 4 in oriented normal Cayley graphs on nonabelian groups are constructed on
the modular maximal cyclic $2$-groups in \S~\ref{sec:modular2gps}.

In the final section \S~\ref{sec:wreath} we show that if an oriented normal Cayley graph on a group $G$ admits PST, then there is an associated graph
admitting PST on the wreath product $G\wr S_n$,
which provides a general method to construct nonabelian examples, including examples of MST on sets of size 4
on nonsolvable groups.

\section{Oriented graphs in Bose-Mesner algebras}\label{sec:BM}

An {\it association scheme} is a set $\cA=\{A_0, \ldots, A_d\}$ of $v\times v$ $01$-matrices satisfying
\begin{enumerate}[i.]
\item
$A_0=I$,
\item
$\displaystyle \sum_{j=0}^d A_j = J$, the matrix of all ones,
\item 
$A_j^T \in \cA$, for $j=0,\ldots, d$,
\item
$A_i A_j = A_jA_i\in \BM$, for $i,j = 0,\ldots,d$.
\end{enumerate}
We call $\BM$ the {\it Bose-Mesner algebra of $\cA$}, which is a commutative algebra closed under matrix multiplication, the Schur product and matrix transposition.   In particular, there exist integers $k_0,\ldots, k_d$ such that
$A_j J = JA_j = k_jJ$ for $j=0,\ldots, d$.  Further, if $j\neq 0$
then $A_j$ has zero diagonal entries.

For $a\in V(X)$, recall that
\begin{equation*}
	\cS_a = \left\{ b \in V(X) \ \vert \ U(\tau)a = \lambda b, \text{for some $\tau > 0$ and $\vert\lambda\vert=1$}\right\}.
\end{equation*}
Define the relation $\equiv $ on $V(X)$ such that for $a,b\in V(X)$,
\begin{align*}
	a\equiv b \mbox{ if and only if PST occurs at a certain time $\tau>0$ from $a$ = $b$}. 
\end{align*}
Clearly, $\equiv$ is an equivalence relation on $V(X)$, and the equivalence class containing $a$ is the set $\mathcal{S}_a$.

\begin{lemma}\label{lem:BMalgP}
 Let $X$ be an oriented graph whose adjacency matrix $A$ belongs to the Bose-Mesner algebra of $\cA$.  
 If PST occurs from $a$ to $b$ where $a\neq b$ at time $\tau$, then $U(\tau) \in \cA$ is the permutation 
 matrix of a fixed-point-free automorphism of $X$.
\end{lemma}
\begin{proof}
    First observe that $U(t) \in \BM$, for any time $t$. If PST occurs from $a$ to $b$ at time $\tau$ then the $(b,a)$-entry is the only non-zero entry in column $a$ of $U(\tau)$.
    We see the $U(\tau) = \lambda A_j$ for some $A_j \in \cA$ with $k_j=1$.  Since $(A_j)_{a,a}=0$, $j\neq 0$ and $A_j$ is a fixed-point-free permutation.  

    Now $A$ is a sum of matrices of the form $A_i - A_i^T$.  Since $A_iJ=A_i^TJ$, the vector of all ones ${\bf 1}$ is an eigenvector of $A$ with eigenvalue $0$.
    Whence
    \begin{equation*}
      \lambda {\bf 1} =  \lambda A_j {\bf 1} = U(\tau) {\bf 1}= e^{0\tau} {\bf 1},
    \end{equation*}
    and $\lambda=1$. 
    Since $A_j$ commutes with $A$, $A_j$ is the permutation matrix of an automorphism of $X$.
\end{proof}

\begin{theorem}\label{thm:BMalgCyclicP}
    Let $X$ be an oriented graph with adjacency matrix $A$ belongs to the Bose-Mesner algebra of some association scheme $\cA$.
    Let $\tau>0$ be the smallest time at which PST occurs from $a$ to another vertex in $X$.
    Then, the following statements hold.
    \begin{enumerate}[(a)]
    	\item $U(\tau)$ is an automorphism of $X$ that acts cyclically on $\cS_a$. 
    	\item The matrix $U(\tau)$ has order $|\cS_a|$.
    	\item For any two vertices $a,b\in V(X)$ such that $a \not\equiv b$, we have $|\cS_a| = |\cS_b|$
    \end{enumerate}
        
\end{theorem}
\begin{proof}
(a) Let $\cS_a =\{u_0, u_1, \ldots, u_m\}$ where $u_0=a$.
For $j=1,\ldots, m$, let $\tau_j$ be the minimum PST time from $a$ to $u_j$, and let $\tau_0=0$.  Without loss of generality, we assume 
\begin{equation*}
    \tau_0<\tau_1 < \cdots < \tau_m.
\end{equation*}  
Note that $\tau_m$ is less than the minimum period of $a$. 

For $j=1,\ldots,m$, PST from $u_{j-1}$ to $u_j$ occurs at time $(\tau_j-\tau_{j-1})$, hence 
\begin{equation*}
    U\left(\tau_j-\tau_{j-1}\right) = \lambda_j P_j,
\end{equation*}
for some fixed-point-free permutation matrix $P_j \in \cA$ satisfying 
$P_ju_{j-1}=u_j$ and phase $\lambda_j$.

Suppose there exists $k<m$ such that $\tau_1=\tau_j-\tau_{j-1}$,
for $j=1,\ldots,k$, and $\tau_1\neq \tau_{k+1}-\tau_k$.
Then $P_1=\cdots=P_k$ and $P_{k+1}$ are distinct permutation matrices in $\cA$ and $P_1 u_j \neq P_{k+1}u_j$, for $j=0,\ldots,m$.

If $\tau_{k+1}-\tau_k<\tau_1$ then 
\begin{equation*}
    U(\tau_{k+1}-\tau_k) a = \lambda_{k+1}P_{k+1}a.
\end{equation*}
So PST occurs from $a$ to $P_{k+1}a$ at time strictly less than $\tau_1$, which contradicts our choice of $u_1$.

If $\tau_{k+1}-\tau_k > \tau_1$ then 
\begin{equation*}
U(\tau_1+\tau_k) a = U(\tau_1) U(\tau_k)a = \lambda_1\lambda_k P_1u_k
\end{equation*}
and PST occurs from $a$ to $P_1u_k\neq P_{k+1}u_k (= u_{k+1})$ at time $(\tau_1+\tau_k)$.
Since $\tau_1+\tau_k < \tau_{k+1}$ is less than the minimum period of $a$, $P_1u_k \neq a$.
Further 
\begin{equation*}
   P_1u_k \not\in \{u_1, \ldots,u_k\} 
\end{equation*} 
because $P_1$ is a permutation matrix satisfying $P_1u_{j-1}=u_j$,
for $j=0,\ldots,k$.  The occurrence of PST from $a$
to  $P_1u_k \not \in \{u_0,\ldots, u_{k+1}\}$ contradicts our choice of $u_{k+1}$.

We conclude that $P_1=\cdots = P_m$, and when restricted to $\cS_a$, this permutation is
the cycle $\left(u_0, u_1, \ldots, u_m\right)$.

(b) Let $P = U(\tau)$. We have seen previously that $P \in \cA$ is a permutation matrix. Let $k$ be the multiplicative order of $P$. When restricted onto $\cS_a$, we know from (a) that the permutation corresponding to $P$ is an $|\cS_a|$-cycle. Hence, $P^{|\cS_a|}$ has diagonal entries equal to $1$. Since $P \in \cA$, we must have that $P^{|\cS_a|} = I$. In other words, the multiplicative order of $P$ is equal to $|\cS_a|$, for $a\in V(X)$. (c) follows immediately from the latter.
\end{proof}

The first part of the following result is a special case of Theorem~6.4.3 in \cite{lato2019quantum}.

\begin{cor}\label{cor:BMalgSize}
    Let $X$ be an oriented graph with adjacency matrix $A$ belongs to the Bose-Mesner algebra of some association scheme $\cA$.
    If PST occurs from $a$ to another vertex in $X$ at time $\tau$ then $\vert \cS_a\vert \in\{2,3,4,6\}$.
    {Moreover, $\tau \in \frac{\pi}{\sqrt{3}} \bb{Q}$ if $\vert \cS_a\vert \in \{3,6\}$, and $\tau \in \pi \bb{Q}$ if $\vert \cS_a\vert =4$.}
\end{cor}
\begin{proof}
Let $A=\sum_{r} \theta_r E_r$ be the spectral decomposition of $A$.  
Since $A\in \BM$, the eigenvalue support of $a$, $\Phi_a$, contains all eigenvalues of $A$.
By Theorem~6.1 of \cite{godsil2020perfect}, there exists a square-free positive integer $\Delta$ such that $\Phi_a \subset {\bb{Z}}(i\sqrt{\Delta})$.

From Theorem~\ref{thm:BMalgCyclicP}, $U(\tau)=P$ is a permutation matrix of order $m=\vert \cS_a\vert$.  For $r=0,\ldots,d$, $ PE_r = \omega_r E_r$ for some $m$-th root of unity $\omega_r$.  Since $P$ is a permutation matrix, it follows from \cite[Lemma~2.10.3]{lato2019quantum}  that $\omega_r \in \bb{Q}(i\sqrt{\Delta})$.  Hence $m \in \{2,3,4,6\}$.

{
Suppose $m\in \{3, 6\}$.  There exists $r$ such that $\omega_r$ is a primitive $m$-th root of unity, so $\Delta=3$.
It follows from $PE_r = \omega_rE_r$ and Theorem~2.10.3 of \cite{lato2019quantum} that $\sqrt{3}i \in \bb{Q}(\theta_r)$.  Hence $\theta_r \in \sqrt{3}i \bb{Z}$, and $\tau \in \frac{\pi}{\sqrt{3}} \bb{Q}$
by Corollary~6.1.3 of \cite{lato2019quantum}.

Similarly, when $m=4$, there exists $r$ such that $PE_r = i E_r$ which implies $i\in \bb{Q}(\theta_r)$.
Thus $\theta_r \in i\bb{Z}$ and $\tau \in \pi\bb{Q}$.
}
\end{proof}

\begin{cor}
    Let $X$ be an oriented graph whose adjacency matrix $A$ belongs to the Bose-Mesner algebra of some association scheme $\cA$.
    If PST occurs in $X$ at time $\tau$ then there exist $a_1, a_2, \ldots, a_s$ such that $V(X)$ is the disjoint union of $\cS_{a_1},\cdots, \cS_{a_s}$, and $\vert \cS_{a_j} \vert = \vert V(X)\vert /s$, for $j=1,\ldots,s$.
\end{cor}
\begin{proof}
    Let $m$ be the order of the permutation matrix $P=U(\tau)$ as in the proof of Theorem~\ref{thm:BMalgCyclicP}.  Since $P \in \cA$, up to re-ordering of vertices, $P$ is a block diagonal matrix in which every block is a $m\times m$ permutation matrix of order $m$.   Let $a_j$ be a vertex in the $j$-th block and the result follows.
\end{proof}

\section{Oriented normal Cayley graphs}\label{sec:Cayley}
The Cayley digraph $\mrm{Cay}(G, C)$ on a group $G$ with a connection set $C$ is the graph with $G$ as its vertex set and $\{(g,\ cg)\ :\ c\in C\}$ as the set of its arcs. We note that $\mrm{Cay}(G, C)$ is oriented if and only if $C \cap C^{-1}=\emptyset$.

Let $\cC_0, \cC_1, \ldots,\cC_d$ be the conjugacy classes of $G$ with 
$\cC_0=\{e\}$.  For $j=0,\ldots,d$, let $A_j$ be the adjacency matrix of the Cayley graph $\mrm{Cay}(G,\cC_j)$.  Then $\cA=\{A_0, A_1,\ldots, A_d\}$ is an association scheme, called the {\it conjugacy class scheme of $G$}.  
A Cayley graph $\mrm{Cay}(G, C)$ is a {\it normal} Cayley graph if 
its connection set $C$ is a union of conjugacy classes of $G$. 
That is, its adjacency matrix belongs to $\BM$.

Suppose PST occurs between two vertices in an oriented normal Cayley graph $\mrm{Cay}(G, C)$.
By vertex transitivity, we may assume without loss of generality that PST occurs from the identity $e$ of $G$.  
Let $\tau >0$ be the minimum time at which PST occurs from $e$ to some vertex $z$.
By Lemma~\ref{lem:BMalgP}, $U(\tau)$ is the adjacency matrix of $\mrm{Cay}(G, \{z^{-1}\})$.
Since $U(\tau) \in \cA$, $\{z^{-1}\}$ is a conjugacy class of $G$ which implies $z$ is a central element of $G$.  It follows from Theorem~\ref{thm:BMalgCyclicP} and Corollary~\ref{cor:BMalgSize} that $\cS_e = \langle z\rangle$ and $z$ has order $2$, $3$, $4$ or $6$.

\begin{theorem}\label{thm:cayleysatset}
    If $\mrm{Cay}(G, C)$ is an oriented normal Cayley graph admitting PST, then $\cS_e$ is a cyclic subgroup of $Z(G)$ that has order $2$, $3$, $4$, or $6$.
\end{theorem}

Given $\chi \in \irr(G)$, let $E_{\chi}$ denote the $G \times G$ matrix with $E_{\chi}(g,h)= \dfrac{\chi(hg^{-1})\chi(e)}{|G|}$. The matrix $E_{\chi}$ is an idempotent projector onto the $\chi$-isotypic subspace of the group algebra $\bb{C}[G]$.
Given a subset $S\subset G$, by $A_{S}$, denote the $G \times G$ matrix with $A_{S}(g,h)= \delta_{S}(hg^{-1})$, where $\delta_{S}$ denoting the indicator function of $S$ in $G$. In the case $S$ is a conjugacy class, $A_{S}$ satisfies the following decomposition as a linear combination of orthogonal idempotents.
\begin{equation}\label{eq:specdecomconjmat}
A_{S} = \sum\limits_{\chi \in \irr(G)} \dfrac{\chi(S)}{\chi(e)} E_{\chi}.
\end{equation}

Let $X:=\mrm{Cay}(G, C)$ be an oriented normal Cayley graph, that is, $C$ is a union of conjugacy classes and $C \cap C^{-1}=\emptyset$. We observe that $A:=A_{C}- A_{C^{-1}}$ is the adjacency matrix of $X$. Using \eqref{eq:specdecomconjmat}, we observe that the spectrum of $A$ is 
$\{\theta_{\chi}\ :\ \chi \in \irr(G)\}$, with $\theta_{\chi} := \dfrac{\chi(C)-\ov{\chi(C)}}{\chi(e)}$ for all $\chi \in \irr(G)$. We also have 
\begin{equation}\label{eq:idempdecom}
A = \sum\limits_{\chi \in \irr(G)} \theta_{\chi} E_{\chi}.
\end{equation} 
Recall that $U(t)=e^{tA}$ and thus, we have 
\begin{equation}\label{eq:idempdecomtmat}
U(t) = \sum\limits_{\chi \in \irr(G)} e^{t\theta_{\chi}} E_{\chi}.
\end{equation} 

\begin{lemma}\label{lem:PSTequiv}
In an oriented normal Cayley graph,
PST occurs from $a$ to $b$ at time $\tau$ if and only if PST occurs from $e$ to $ba^{-1}$.
Further, $ba^{-1}$ is a central element of $G$.
\end{lemma}
\begin{proof}
Let $P$ be the permutation matrix of the automorphism of $X$ that maps $g$ to $ga^{-1}$.
Then $PU(\tau)=U(\tau)P$, and $U(\tau)a=b$ if and only if $U(\tau)e=ba^{-1}$.  
By Theorem~\ref{thm:BMalgCyclicP}, $ba^{-1} \in \cS_e \subset Z(G)$.   
\end{proof}

\begin{theorem}\label{thm:characterization}
PST occurs  from $e$ to $z$ at time $\tau$ in an oriented normal Cayley graph $\mrm{Cay}(G, C)$ if and only if the following hold.
\begin{enumerate}
\item $z\in Z(G)$.
\item \label{Cond:chi}
For $\chi \in \irr(G)$,
\[\dfrac{\chi(z)}{\chi(e)}=  \exp\left({\tau\ \dfrac{\chi(C)-\ov{\chi(C)}}{\chi(e)}}\right).\]
\end{enumerate}
\end{theorem}  
\begin{proof}
By Theorem~\ref{thm:BMalgCyclicP}, if PST occurs from $e$ to $z$ then $z\in Z(G)$.
From \eqref{eq:idempdecomtmat}, $U(\tau)e = z$ if and only if
\begin{equation*}
  e^{\tau \theta_{\chi}} E_{\chi}e = E_{\chi} z, \quad \text{for $\chi \in \irr(G)$} 
\end{equation*}
which is equivalent to
\begin{equation}\label{eq:chiCondition}
  e^{\tau \theta_{\chi}} \chi(g^{-1}) = \chi(zg^{-1}), \quad \text{for $\chi \in \irr(G)$ and $g\in G$.}
\end{equation}
Since $z \in Z(G)$, $z$ acts on every irreducible module as a scalar matrix
and for $g\in G$
\begin{equation*}
    \dfrac{\chi(zg^{-1})}{\chi(g^{-1})} = \dfrac{\chi(z)}{\chi(e)}.
\end{equation*}
Condition~(\ref{Cond:chi}) follows immediately from \eqref{eq:chiCondition}.
\end{proof}
From the above, it follows that if PST occurs at a time $\tau$ from $e$ to $z$, then PST occurs at the time $n\tau$ from $e$ to $z^n$. This helps us conclude the following result.
\begin{cor}\label{cor:multistate}
Let $G$ be a finite group, $X:=\mrm{Cay}(G, C)$ be a normal oriented Cayley graph, and $z$ be a central element satisfying Condition~(\ref{Cond:chi}) in Theorem~\ref{thm:characterization}. Then $X$ has multiple state transfer on the subgroup $\left\langle z \right\rangle$.
\end{cor} 

We end this section with a non-existence result. Before stating it, we recall some standard notation. Given a character $\chi$, let $\bb{Q}(\chi)$ denote the field generated 
by its values and let $\mrm{ker}(\chi)$ denote its kernel. 
\begin{cor}\label{cor:non-exist}
 Let $G$ be a finite group and $z \in Z(G)$ be such that there exists a subset $Y\subset \irr(G)$ satisfying:
 (i) $z \notin \mrm{ker}(\chi)$, for all $\chi \in Y$; and (ii) $\bigcap\limits_{\chi \in Y} \bb{Q}(\chi)= \bb{Q}$. Then no normal Cayley graph on $G$ admits multiple state transfer on $\left\langle z \right\rangle$. 
\end{cor}
\begin{proof}
By Corollary~\ref{cor:multistate}, it suffices to show that there is no normal Cayley graph on $G$ admits PST from $e$ to $z$. We assume the contrary, that is, there is a normal Cayley graph $\mrm{Cay}(G,C)$ that has PST from $e$ to $z$, at a time $\tau$.  Using Condition~(\ref{Cond:chi}) of Theorem~\ref{thm:characterization}, it follows that
\begin{subequations}
    \begin{align*}
        \dfrac{\chi(z)}{\chi(e)} &=e^{\tau \theta_{\chi}},
    \end{align*}
\end{subequations}
for all $\chi \in Y$.
Using (i) of the hypothesis, since $z\notin \mrm{ker}(\chi)$, for all $\chi \in Y$, it follows from the above equation that $\theta_{\chi}\neq 0$ for all $\chi \in Y$.  By \cite[Theorem 7.1]{godsil2020perfect}, the vertex $e$ is periodic, and thus by \cite[Theorem~6.1]{godsil2020perfect}, it follows that for all $\chi,\psi \in Y$, we have, $\theta_{\chi}/\theta_{\psi} \in \bb{Q}$. Thus, since $\theta_{\psi}\in \bb{Q}(\psi)$,  we have $\theta_{\chi} \in \bigcap\limits_{\psi \in Y} \bb{Q}(\psi) =\bb{Q}$. Since $\theta_{\chi}$ is purely imaginary, it follows that $\theta_{\chi}=0$, which is a contradiction.  
\end{proof}


\section{Oriented normal Cayley graphs on solvable groups}\label{sec:solvable}
We now show that $\vert \cS_e \vert \neq 6$ in oriented normal Cayley graphs on solvable groups.
\begin{theorem}\label{thm:solvableMST}
    If $G$ is a solvable group, $\mrm{Cay}(G, C)$ is an oriented normal Cayley graph admitting PST, then $\cS_e$ is a cyclic subgroup of $Z(G)$ that has order either $2$, or $3$, or $4$.
\end{theorem}
\begin{proof}
  Assume the contrary. Then, by Theorem~\ref{thm:cayleysatset}, there is a
  solvable group $G$ of order $6n$ and $C \subset G$ such that $\mrm{Cay}(G, C)$ is an oriented normal Cayley graph, with $\cS_{e}$ equal to a cyclic subgroup $Z$  of order $6$ in $Z(G)$.

  We shall apply the theory of $p$-special characters, as developed (in far greater generality) in \cite{Gajendragadkar}, for the primes $p=2$ and $3$. For this purpose we recall some definitions. Given a character $\phi$ of a group $H$, we define
  $\det(\phi)(h)$ to be the determinant of $T(h)$, where $T$ is
  any matrix representation affording $\phi$. This is a well-defined linear character of $H$.  Let $p$ be a prime. An irreducible
  character $\chi$ is {\it $p$-special} if $\chi(1)$ is a power of $p$
  and the following further conditions hold: for every subnormal subgroup
  $H$ of $G$ and every irreducible constituent $\phi$ of $\chi_H$,
  the order of $\det(\phi)$ is a power of $p$.

  Let $\alpha$ and $\beta$ be characters of $Z$ of order $2$ and $3$
  respectively, By \cite[Cor. 4.8]{Gajendragadkar}, $G$ has a $2$-special
  irreducible character $\chi$ such that $\chi_Z$ is a multiple of $\alpha$.
  Likewise, $G$ has a $3$-special irreducible character $\psi$ such that
  $\psi_Z$ is a multiple of $\beta$. Hence
  by Condition~(\ref{Cond:chi}) in Theorem~\ref{thm:characterization} we have $\theta_\chi\neq 0$ and $\theta_\psi\neq 0$.
  
  By \cite[Prop. 6.3(a)]{Gajendragadkar},  $\mathbb Q(\chi)$ lies in
  a cyclotomic field of order equal to the  $2$-power in the exponent of $G$,
  and $\mathbb Q(\psi)$ lies in
  a cyclotomic field of order equal to the  $3$-power in the exponent of $G$. The intersection of these
  cyclotomic fields is $\mathbb Q$. Hence Corollary~\ref{cor:non-exist} applies.
\end{proof}

\section{Oriented Cayley graphs on finite abelian groups}\label{sec:abelian}

In this section, we consider PST
in oriented normal Cayley graphs on finite abelian groups $\bb{Z}_r^n$.  By Theorem~\ref{thm:cayleysatset} and Theorem~\ref{thm:solvableMST}, if there is PST
we must have $\abs{\cS_e}\in \{2,3,4\}$. As Cayley graphs in the groups $\bb{Z}_2^n$ are not oriented,
it is natural to first consider the groups ${\mathbb Z_3}^n$ and ${\mathbb Z_4}^n$ . However, we point out that MST is known to exist for other abelian groups. For example, $\mrm{Cay}(\mathbb Z_8,\{1,2,5\})$ has MST on a set of size 4. (See \cite[10.2. Example 6]{godsil2024oriented}.)

We consider the elements of this group to be column vectors of length $n$, with entries in $\bb{Z}_{r}$. The characters of $\bb{Z}_{r}^{n}$ are indexed by its elements and are of the form $\chi_{v}: \bb{Z}_{r}^{n} \to \bb{C}^{\times}$ (for some $v\in \bb{Z}_{r}^{n}$) satisfying 
\begin{equation*}
    \chi_{v}(w)= \exp\left( \dfrac{2\pi i v^{t}w}{r} \right), \quad \text{for all $w \in \bb{Z}_{r}^{n}$.}  
\end{equation*}

Our results for ${\mathbb Z_3}^n$ and ${\mathbb Z_4}^n$
are inspired by the following result from \cite{CheungGodsil} about undirected Cayley graphs on  ${\mathbb Z_2}^n$.
\begin{theorem}
For $C \subset \bb{Z}_2^n \backslash \{\mbf{0}\}$, let $\sigma$ be the sum of the elements in $C$.  Then 
\begin{equation*}
    U\left(\frac{\pi}{2}\right) \mbf{0} = \sigma.
\end{equation*}
That is, PST occurs  in $\mrm{Cay}(\bb{Z}_2^n, C)$ at time $\frac{\pi}{2}$ if and only if $\sigma \neq \mbf{0}$.
\end{theorem}
We extend the above result to $\mrm{Cay}\left(\bb{Z}_r^n, C\right)$ for $r=3, 4$.
For $v\in \bb{Z}_r^n$ and $j=0,\ldots,r-1$, let
\begin{equation*}
    n_{v,j} = \left\vert \left\{w \in C \ :\ v^Tw = j \pmod{r}\right\}\right\vert.
\end{equation*}
Then
\begin{equation}\label{eq:theta_chi}
  \chi_{v}(C)-\ov{\chi_{v}(C)} =\sum\limits_{ w \in C} \left(\exp\left( \dfrac{2\pi i v^{t}w}{r} \right)- \exp\left(-\dfrac{2\pi i v^{t}w}{r}\right) \right)= 2 i \sum_{j=1}^{r-1}  n_{v,j} \sin\left(\dfrac{2j \pi}{r}\right).  
\end{equation}
\begin{theorem}
Let $C \subset \bb{Z}_{3}^{n} \backslash \{\mbf{0}\}$ be a connection set with $C \cap C^{-1}=\emptyset$ and $\sigma:=\sum\limits_{w \in C} w$. 
\begin{equation*}
    U\left(\frac{2\pi}{3\sqrt{3}}\right) \mbf{0} = \sigma.
\end{equation*}
That is, PST occurs in $\mrm{Cay}(\bb{Z}_3^n, C)$ from $\mbf{0}$ to $\sigma$ at time $\frac{2\pi}{3\sqrt{3}}$ if and only if $\sigma \neq \mbf{0}$. 
Moreover, if $\sigma \neq \mbf{0}$, then $\mrm{Cay}(\bb{Z}_3^n, C)$ has MST on $\langle \sigma\rangle$.
\end{theorem}  
\begin{proof}
From \eqref{eq:theta_chi}, we have
\begin{align*}
\chi_{v}(C)-\ov{\chi_{v}(C)} =  \left( n_{v,1}-n_{v,2} \right) \sqrt{3} i,
\end{align*}
and
\begin{align*}
\chi_{v}(\sigma)  =\exp\left( \dfrac{2\pi i}{3} \sum\limits_{w \in C } v^{t}w \right) 
 = \exp\left( \dfrac{2\pi i \left( n_{v,1}+2n_{v,2} \right) } {3} \right). 
\end{align*}
The above two equalities imply that
\[\chi_{v}(\sigma) = \exp\left( \dfrac{2\pi}{3\sqrt{3}} \left(\chi_{v}(C)-\ov{\chi_{v}(C)}\right) \right),\] 
for all $v \in \bb{Z}_{3}^{n}$. Now, the result follows from Theorem~\ref{thm:characterization} and Corollary~\ref{cor:multistate}. 
\end{proof}

The situation for $\bb{Z}_4^n$ is a little different.
The next two results show that we may have PST between two elements, but there cannot be MST on a set of size 4.

\begin{theorem}
Let $C \subset \bb{Z}_{4}^{n} \backslash \{\mbf{0}\}$ be a connection set with $C \cap C^{-1}=\emptyset$ and $\sigma:=\sum\limits_{w \in C} w$. \begin{equation*}
    U\left(\frac{\pi}{2}\right) \mbf{0} = 2\sigma.
\end{equation*}
That is, PST occurs in $\mrm{Cay}(\bb{Z}_4^n, C)$ from $\mbf{0}$ to $2\sigma$ at time $\frac{\pi}{2}$ if and only if $\sigma$ has order 4. 
\end{theorem}  
\begin{proof}
We have
\begin{align*}
\chi_{v}(C)-\ov{\chi_{v}(C)} =  2i\left( n_{v,1}-n_{v,3} \right) ,
\end{align*}
and
\begin{align*}
\chi_{v}(2\sigma)  =\exp\left( \dfrac{4\pi i}{4} \sum\limits_{w \in C } v^{t}w \right) 
 = \exp\left( \pi i \left( n_{v,1}+3n_{v,3} \right)  \right). 
\end{align*}
The above two equalities imply that
\[\chi_{v}(2\sigma) = \exp\left( \dfrac{\pi}{2} \left(\chi_{v}(C)-\ov{\chi_{v}(C)}\right) \right),\] 
for all $v \in \bb{Z}_{4}^{n}$. Now, the result follows from Theorem~\ref{thm:characterization}. 
\end{proof}

\begin{theorem}\label{thm:NoOrder4}

    An oriented Cayley graph on $\bb{Z}_4^n$ does not have MST on a set of size 4.
\end{theorem}
\begin{proof}
Let $\langle z\rangle$ be a cyclic subgroup of order $4$. Let $\lambda$ be the character of $\langle z\rangle$ such that $\lambda(z)=-1$. There is a subgroup $A$ with ${\mathbb Z_4}^n=\langle z\rangle\times A$. Thus, $\lambda$ can be extended to a character of ${\mathbb Z_4}^n$  whose values are rational. The result now follows Corollary~\ref{cor:non-exist}.
\end{proof}

\section{Oriented normal Cayley graphs on extraspecial $3$-groups}\label{sec:extraspecial}
In this section, we construct examples of multiple state transfer in Cayley graphs over extraspecial $3$-groups. An extraspecial $p$-group is a $p$-group $G$ whose centre has order $p$ and $G/Z$ is isomorphic to an elementary abelian $p$-group. We now recall some facts about extraspecial $p$-groups and their complex characters. Background on these groups and their representations can be found in  \cite[\S~5.5]{gorenstein2007finite}. The size of an extraspecial $p$-group must be $p^{2n+1}$ for some integer $n\geq 1$. The conjugacy class containing an element $g \in G \setminus Z$ is the coset $gZ$. Irreducible characters of $G$ are either linear or are of degree $p^{n}$. Linear characters of $G$ factor through its derived subgroup $Z$ and so are the characters of the elementary abelian group $G/Z= \bb{Z}_{p}^{2n}$. Irreducible non-linear characters are indexed by non-trivial irreducible characters of $Z$ and are of the form 
\[\psi_{\lambda} = \begin{cases}
0 & \text{if $g \in G\setminus Z$} \\
p^{n}\lambda(g) & \text{if $g \in Z$},
\end{cases}
\] 
where $\lambda$ is a non-trivial linear character of $Z$.

Let $G$ be an extraspecial $3$-group of size $3^{2n+1}$. Let $Z\cong \bb{Z}_{3}$ denote its centre and consider a generator $z$ of $Z$. Let $F$ denote a basis of $G/Z \cong \bb{Z}_{3}^{n}$ and let $E$ denote the pre-image of $F$ with respect to the quotient map $G \to G/Z$. The set $C:=E \cup \{z\}$ is a union of conjugacy classes and we have $C \cap C^{-1}=\emptyset$. We show that the oriented normal Cayley graph $X := \mrm{Cay}(G,\ C)$ has multiple state transfer on $Z$. This is done using Corollary~\ref{cor:multistate}.

Given a non-linear character $\psi_{\lambda}$ of $G$, we have $\psi_{\lambda}(C)-\ov{\psi_{\lambda}(C)}=3^{n} (\lambda(z)- \ov{\lambda(z)})$. As $\lambda$ is a non-trivial character, $\lambda(z)$ must be a third root of unity, and thus, \[\dfrac{\psi_{\lambda}(C)-\ov{\psi_{\lambda}(C)}}{\psi_{\lambda}(1)}= \begin{cases} \sqrt{3}i\ &\text{if $\lambda(z)= e^{2\pi i/3}$}\\
-\sqrt{3}i &\text{if $\lambda(z)= e^{-2\pi i/3}$}.
\end{cases}\]

Setting $\tau=\dfrac{2 \pi}{3 \sqrt{3}}$, we now observe that $z$ and $C$ satisfy (1) and (2) of Theorem~\ref{thm:characterization}, for all non-linear irreducible characters of $G$. We now show that the same is true for linear characters.

Recall that the linear characters of $G$ are the same as those of $G/Z = \bb{Z}_{3}^{2n}$. Non-trivial linear characters of $\bb{Z}_{3}^{2n}$ are indexed by non-zero vectors of $\bb{Z}_{3}^{2n}$. The character $\psi_{v}$ indexed by a vector $v$ satisfies
\[\psi_{v}(x)= e^{\dfrac{i2\pi v^{t}x}{3}},\] for all $x\in \bb{Z}_{3}^{2n}$. The character $\psi_{v}$ 
defines a character of $G$ by identifying $\psi_{v}=\psi_{v} \circ \pi$, where $\pi: G \to G/Z$ is the standard quotient map. We recall that $C= \{z\}\cup E$, where $E=\pi^{-1}(F)$ and $F$ is a basis for $G/Z$ as a vector space over $\bb{Z}_{3}$.
 We have 
\begin{align*}
\psi_{v}(C)- \ov{\psi_{v}(C)} & = 3 \left( \sum\limits_{x \in F} e^{\dfrac{i2\pi v^{t}x}{3}} - e^{\dfrac{-i2\pi v^{t}x}{3}} \right),
\end{align*} 
and thus, \[\tau \dfrac{\psi_{v}(C)- \ov{\psi_{v}(C)}}{\psi_{v}(1)}= \dfrac{2\pi}{3 \sqrt{3}} \left( \psi_{v}(C)- \ov{\psi_{v}(C)} \right) =i2m \pi, \] for some integer $m$. Since $\psi_{v}(z)/\psi_{v}(z)=1$, we see that condition (2) of Theorem~\ref{thm:characterization} is also satisfied by linear characters of $G$. Now, by Corollary~\ref{cor:multistate}, we have multiple state transfer on $Z$.


\section{Multiple state transfer on sets of size 3 
on oriented normal Cayley graphs of some nonsolvable groups}\label{sec:nonsolvable}
As the conditions of Theorem~\ref{thm:characterization}
can be read off the character table of the group, it is
possible to search for examples by using the computer algebra system GAP \cite{GAP4}, which contains character tables of
many groups and facilitates the extraction and manipulation of the data contained in them. There are many examples of oriented normal Cayley graphs of nonsolvable (even quasisimple) groups that admit MST of size 3.
For illustration we shall describe two such examples, one on a small group, the triple cover of $A_6$, and one on a much larger group, the triple cover
of Janko's third sporadic group $J_3$.

Let $G$ be the unique perfect group of order 1080. This is the unique nonsplit central extension of the alternating group $A_{6}$ by a group of order 3. We state some facts about the group, which can be verified using a computer algebra system such as GAP \cite{GAP4}. This group has exactly two classes of elements of order $12$, two classes of elements of order $6$, and two central elements of order $3$. Moreover, if $g$ is an element of order $12$, $g^{4}$ is in the centre. Pick an element $g$ of order $12$ and set $z:=g^{4} \in Z(G)$. Let $h$ be an element of order $6$ which is not conjugate to $g^2$. We define $C$ to be the set of elements conjugate to an element in  $\{z,g, h\}$. Since $G$ is quasisimple (i.e. $G$ is perfect and $G/Z(G)$ is simple) and $C\nsubset Z(G)$, it is immediate that $C$ generates $G$, so $\mrm{Cay}(G,\ C)$ is connected.
Using the GAP and Theorem~\ref{thm:characterization}, we can verify that in the graph $\mrm{Cay}(G,\ C)$, PST occurs from $1$ to $z$, at a time $\tau = 2\pi/(3\sqrt{3})$. By Corollary~\ref{cor:multistate}, we have multiple state transfer on $Z(G)=\langle z \rangle\cong \bb{Z}_{3}$.

Let $G=3.J_3$, the triple cover of $J_3$, following
the notation in GAP and throughout the literature. 
Choose $z\in Z(G)$ of order 3. Next choose any element
$g$ of order 30. (There are 4 conjugacy classes of
elements of order 30.) Then $g^7$ is also of order 30 
and turns out not to be conjugate to $g$.
We take $C$ to be the union of
the conjugacy classes of $g$, $g^7$ and $z$.
The connectedness  of $\mrm{Cay}(G,\ C)$ follows, as in the previous example, from the quasisimplicity of $G$.
Using GAP we can check  that the conditions of Theorem~\ref{thm:characterization} hold at time
$\tau = \pm2\pi/(3\sqrt{3})$ (depending on the choice of $z$), so by Corollary~\ref{cor:multistate}, we have MST on $Z(G)$. This example is one of many
for the group $3.J3$ and was chosen for its ease of description.

\section{Multistate transfer in modular maximal-cyclic groups}\label{sec:modular2gps}
Given a positive integer $n$, the modular maximal-cyclic group of order $2^n$, denoted by $M_{2}(n)$, is a group defined by the following presentation:
\[M_{2}(n) =\left\langle x,\sigma\ \big\lvert \ x^{2^{n-1}}=\sigma^2=e,\ \sigma x \sigma =x^{2^{n-2}+1} \right\rangle.\] It is well known and elementary to observe that $\langle x \rangle \trianglelefteq M_{2}(n)$, $M_{2}(n) = \langle x \rangle \rtimes \langle \sigma \rangle $, $Z(M_{2}(n))= \langle x^2 \rangle \cong \bb{Z}_{2^{n-2}}$, and $[M_{2}(n),\ M_{2}(n)] =\langle x^{2^{n-2}} \rangle \cong \bb{Z}_{2}$. In this section, we construct an oriented normal Cayley graph on $M_{2}(n)$, which admits MST on a set of $4$ vertices. 

We now describe the conjugacy classes of $M_{2}(n)$. 
These are well known and can be derived easily. As mentioned above, even powers of $x$ are central. Given any pair of non-commuting elements $h,k$, we have $hkh^{-1}k^{-1} \in [M_{2}(n),\ M_{2}(n)]\setminus \{e\} =\{x^{2^{n-2}}\}$, so $hkh^{-1}=x^{2^{n-2}}k$. Thus, the conjugacy class of  any non-central element $h \in M_{2}(n)$ is the two-element set $\{h,x^{2^{n-2}}h\}$.

We now describe the characters of $M_{2}(n)$. The linear characters are exactly the same as those of its abelianization, $M_{2}(n)/\langle x^{2^{n-2}}\rangle \cong \bb{Z}_{2} \times \bb{Z}_{2^{n-2}}$. We need the following result about non-linear irreducible characters. Given $\psi \in \irr(\langle x \rangle)$, we define $\psi^{\sigma} \in \irr(M_{2}(n))$ to be the character satisfying $\psi^{\sigma}(y)=\psi(\sigma y\sigma)=\psi(y^{2^{n-2}+1})$, for all $y \in \langle x \rangle$.

\begin{lemma}\label{lem:nonlin}
 Given a non-linear character $ \chi \in \irr(M_{2}(n))$, then (i) $\chi(g)=0$, for all $g \notin Z(M_{2}(n))$; and (ii) $\chi(x^{2^{n-3}})=\pm 2i$.     
\end{lemma}
\begin{proof}
 By Clifford theory (\cite[Corollary~6.19]{isaacs1994character}, since $[M_{2}(n): \langle x\rangle]=2$, given any non-linear irreducible character $\chi$, its restriction $\chi|_{\langle x \rangle}$ is a sum of distinct conjugates of a linear character of $\psi \in \irr(\langle x \rangle)$. The set of conjugates of $\psi$ is $\{\psi,\ \psi^{\sigma}\}$ and thus $\chi|_{\langle x \rangle}=\psi +\psi^{\sigma}$. Since $\psi \neq \psi^{\sigma}$, we must have $o(\psi)= 2^{n-1}$. Thus, $\psi(x)$ must be a primitive $2^{n-1}$th root of unity. Therefore, (ii) follows from the fact that $\chi|_{\langle x \rangle}=\psi +\psi^{\sigma}$.

 Given a non-central element $g$, let $h \in M_{2}(n)$ be such that $hgh^{-1}g^{-1} \neq e$. Since $\{e,\ x^{2^{n-2}}\}$ is the commutator, it follows that $hgh^{-1}g^{-1}=x^{2^{n-2}}$, and thus $hgh^{-1}=x^{2^{n-2}}g$. Therefore, 
 \begin{equation}\label{eq:noncentral}
 \chi(g)=\chi(hgh^{-1})=\chi(x^{2^{n-2}}g) =\dfrac{\chi(x^{2^{n-2}})}{\chi(e)} \chi(g),
 \end{equation}
 with the last equality following from the fact that $x^{2^{n-2}}$ is a central element. As $x^{2^{n-2}}$ is an element of order $2$ and $\psi$ is faithful, it follows that $\psi(x^{2^{n-2}})=-1$, and thus $\chi(x^{2^{n-2}})=-2$. Using \eqref{eq:noncentral}, it follows that
 $\chi(g)=-\chi(g)$, and thus (i) is true.
\end{proof}

\begin{theorem}\label{thm:mm-c}
 Let $n\geq 5$, \[M_{2}(n) =\left\langle x,\sigma\ \lvert \ x^{2^{n-1}}=\sigma^2=e,\ \sigma x \sigma =x^{2^{n-2}+1} \right\rangle, \]
and let 
\[ C= \{x^{2^{n-3}},\ x^{2^{n-4}}\sigma,\ x^{2^{n-2}+2^{n-4}}\sigma \} \bigcup \{x^{4k+1}\ :\ 0\leq k \leq 2^{n-3}-1 \}\]
 Then $\mrm{Cay}(M_{2}(n), C)$ is an oriented normal Cayley graph on $M_{2}(n)$, which admits MST on $\left\langle x^{2^{n-3}} \right\rangle$.   
\end{theorem}
\begin{proof}
Using Corollary~\ref{cor:multistate}, it suffices to show that PST occurs from $e$ to $x^{2^{n-3}}$, at time $\tau =\pi/4$. In other words, using Theorem~\ref{thm:characterization}, it suffices to show that
\begin{equation}\label{eq:mmcchar}
\exp\left(\dfrac{\theta_{\phi}\pi}{4} \right) = \dfrac{\phi(x^{2^{n-3}})}{\phi(1)},
\end{equation}
for all $\phi \in \irr(M_{2}(n))$.

We first show that the  above is true for linear characters. As we mentioned at the start of this section, we have  $[M_{2}(n),\ M_{2}(n)] =\langle x^{2^{n-2}}\rangle$, and thus given a linear character $\chi$ of $M_{2}(n)$, we have $\chi^{2^{n-2}}=1$. By $o(\chi)$, we denote the order of $\chi$, as an element of the dual group $\widehat{M_{2}(n)}$.  

\paragraph{Case 1:} Assume that $o(\chi)=2^{r}$, for some $r\geq 3$. Then, we have $\chi(x^{2^{r-1}})=-1$ and thus $\chi(x^{4(k+2^{r-3})+1})=-\chi(x^{4k+1})$, for all $0\leq k \leq 2^{n-3}-1$. We note that the map $x^{4k+1} \mapsto x^{4(k+2^{r-3})+1}$ is a bijection on $\{x^{4k+1}\ :\ 0\leq k \leq 2^{n-3}-1 \}$. Therefore, we have 
 \begin{align*}
  \sum\limits_{k=0}^{2^{n-3}-1} \chi(x^{4k+1}) & = \sum\limits_{k=0}^{2^{n-3}-1} \chi(x^{4(k+2^{r-3})+1})\\
  & = -\sum\limits_{k=0}^{2^{n-3}-1} \chi(x^{4k+1}),
 \end{align*}
 with the last equality following from the fact that $\chi(x^{4(k+2^{r-3})+1})=-\chi(x^{4k+1})$, for all $k$. Thus, $\sum\limits_{k=0}^{2^{n-3}-1} \chi(x^{4k+1})=0$.
 We now have
\begin{align*}
\theta_{\chi} &= \chi(x^{2^{n-3}}) + 2\chi(x^{2^{n-4}} \sigma) + \sum\limits_{k=0}^{2^{n-3}-1} \chi(x^{4k+1})- \ov{\left(\chi(x^{2^{n-3}}) + 2\chi(x^{2^{n-4}} \sigma) + \sum\limits_{k=0}^{2^{n-3}-1} \chi(x^{4k+1}) \right)} \\
& = \chi(x^{2^{n-3}}) + 2\chi(x^{2^{n-4}} \sigma) -\ov{ \left( \chi(x^{2^{n-3}}) + 2\chi(x^{2^{n-4}} \sigma) \right) }.
\end{align*}
Since $x^{2^{n-2}} \in [M_{2}(n), M_{2}(n)] \leq \mrm{ker}(\chi)$, we have  $\chi(x^{2^{n-2}})=1$, we have  $\chi(x^{2^{n-3}})=\pm 1$, and $\chi(x^{2^{n-4}} \sigma)= \chi(\sigma) \chi(x^{2^{n-4}})=\pm\chi(\sigma)\sqrt{\chi(x^{2^{n-3}})}$. Since $\sigma$ is an element of order $2$, it follows that $\chi(\sigma)=\pm 1$, and thus $\chi(x^{2^{n-4}} \sigma)= \pm \sqrt{\chi(x^{2^{n-3}})}$. Therefore, we have
\begin{align*}
 \theta_{\chi} = \begin{cases}
 0 & \text{if $\chi(x^{2^{n-3}})=1$} \\
  \pm 4 i & \text{if $\chi(x^{2^{n-3}})=-1$,}
 \end{cases}
\end{align*}
and thus \eqref{eq:mmcchar} is satisfied by all linear characters of order at least $8$.

\paragraph{Case 2:} Assume that $o(\chi) =4$. In this case, have $\chi(x)=\pm i$, and thus, 
\begin{align}\label{eq:order4}
     \sum\limits_{k=0}^{2^{n-3}-1} \chi(x^{4k+1})= \pm 2^{n-3} i, \text{when $o(\chi)=4$}. 
 \end{align}
We also have $\chi(x^{2^{n-3}}), \chi(x^{2^{n-4}} \sigma) \in \bb{Q}$. Therefore, we have $\theta_{\chi}= \pm 2^{n-2} i$. As $n\geq 5$, we observe that (a) $\chi(x^{2^{n-3}})=1$ and (b) $(\pi/4)\theta_{\chi}$ must be an even multiple of $\pi$. Therefore \eqref{eq:mmcchar} is satisfied for all linear characters of order $4$.

\paragraph{Case 3:} We now consider the case $o(\chi) =2$. In this case, must have $\chi(x)=\pm 1$ and $\chi(\sigma)=\pm1$. Therefore, $\chi(g) \in \bb{Q}$, and thus, $\theta_{\chi}=0$. Since $\chi(x)=\pm 1$, we have $\chi(x^{2^{n-3}})=1$, and thus \eqref{eq:mmcchar} holds for linear characters of order $2$ as well.

We now move on to non-linear irreducible characters. Using Lemma~\ref{lem:nonlin}, it follows that for a non-linear irreducible $\chi$, we have $\chi(x^{2^{n-3}})/\chi(1)= \pm i$ and thus $\theta_{\chi}= 2\chi(x^{2^{n-3}})/\chi(1)= \pm 2i$, with the same sign.  It follows that \eqref{eq:mmcchar} holds for all non-linear irreducible characters as well. 
\end{proof}

\begin{remark} For $M_{2}(2)$ it is an easy exercise to check that for $C=\{x^2, x\sigma, x^5\sigma\}$ we have MST on the central subgroup $\langle x^2\rangle$ 
of order $4$.
\end{remark}

\section{Wreath products}\label{sec:wreath}
Let $X:=\mrm{Cay}(G, C)$ be an oriented normal Cayley graph that admits PST from $e$ to $z$, at a time $\tau$. In this section, we construct an oriented normal Cayley graph on $G \wr S_{n}$, which admits PST.

We first describe the group $G\wr S_{n}$. By $G^{n}$, we denote the $n$-fold direct product of $G$. We use boldfaced lower case letters to represent elements of $G^n$; and given $\mbf{x} \in G^{n}$ and $i \in G$, we define $x_{i} \in G$ to be such that $\mbf{x}=(x_{1},\ldots,x_{n})$. The symmetric group $S_{n}$ acts on the group $G^{n}$ by permuting the indices of its elements: $\pi\cdot x =\pi \cdot (x_{1},\ldots, x _{n})= (x_{\pi^{-1}(1)}, \ldots, x_{\pi^{-1}(n)})$, for all $\pi\in S_{n}$ and $\mbf{x}\in S_{n}$. The {\it wreath product} $G\wr S_{n}$ is the semidirect product $G^{n} \rtimes S_{n}$ defined by the above action. We express elements of $G\wr S_{n}$ as tuples of the form $(\mbf{x} ;\ \pi)$, where $\mbf{x} \in G^{n}$ and $\pi \in S_{n}$.

We first describe the conjugacy classes and all the (complex) irreducible characters of $G \wr S_{n}:= G^{n} \rtimes S_{n}$. The description of these are well known and can be found in many standard texts on character theory of symmetric groups. For instance, the reader may refer to \cite[Chapter 4]{James_1984} and \cite[Appendix B of Chapter 1]{macdonald1998symmetric}. 

 Consider $\mbf{y}=(y_{1},\ldots,y_{n}) \in G^{n}$ and an $r$-cycle $\kappa\in S_{n}$. Let $a$ be the smallest positive integer such that $\kappa=(a,\kappa(a)\ldots \kappa^{r}(a))$. We define the product (as elements of $G$) 
 \begin{equation}\label{eq:cyclicproduct}
 \mbf{y}_{\kappa} := y_{a}y_{\kappa^{-1}(a)} \cdots    y_{\kappa^{-i}(a)}  \cdots y_{\kappa^{-r+1}(a)}.
 \end{equation}
 This is known as the {\it cycle product} (cf. \cite[4.2.1]{James_1984}) associated with $\kappa$ with respect to $\mbf{y}$. As an example, given $\mbf{g}=(a,b,c) \in G^{3}$, $\sigma =(1,2,3)$, and $\beta=(2,3)$, we have $\mbf{g}_{\sigma}=acb$, $\mbf{g}_{\sigma^{-1}}=abc$, and $\mbf{g}_{\beta}=bc$. It is an easy exercise to observe that $(\mbf{x}, \alpha)\in G\wr S_{n}$ is conjugate to $(\mbf{y},\ \kappa)$ if and only if (a) $\alpha$ is also an $r$-cycle and (b) $\mbf{x}_{\alpha}$ is conjugate to $\mbf{y}_{\kappa}$ in $G$. More generally, given $g=(\mbf{x}; \pi) \in G \wr S_{m}$, with $\pi=\prod \pi_{i}$ as a decomposition of $\pi$ into disjoint cycles, and a conjugacy class $\mfk{K}$ of $G$, define $m_{g,r}(\mfk{K})$ to be the number of $r$-cycles $\kappa$ in $\pi$ such that $\mbf{x}_{\kappa} \in \mfk{K}$. Let $\mrm{ty}_{g}(\mfk{K})$ denote the partition $\left\langle 1^{m_{g,1}(\mfk{K})}, \ldots, r^{m_{g,r}(\mfk{K})}, \ldots \right\rangle$. We have now defined a partition-valued function $\mrm{ty}_{g}$ on  the set $\mrm{Cl}(G)$ of conjugacy classes of $G$, with $\sum\limits_{\mfk{K} \in \mrm{Cl}(G)} |\mrm{ty}_{g}(\mfk{K})| =n$. It is an easy exercise to show that $g,h \in G \wr S_{n}$ are conjugate if and only if $\mrm{ty}_{g}= \mrm{ty}_{h}$. Conversely, given a partition-valued function $P$ on $\mrm{Cl}(G)$ with $\sum |P(\mfk{K})|=n$, there is some $g \in G \wr S_{n}$ such that $P=\mrm{ty}_{g}$. 
 
To describe the irreducible characters of $G \wr S_{n}$ we must  consider  certain irreducible characters of $G\wr S_{m}$, for $m\leq n$. 
Consider $\phi \in \irr(G)$ and let $V$ be the $G$-module affording $\phi$ as its character. Then $\bigotimes^{m} V$ is a $G^{m}$-module affording $\otimes^{m} \phi \in \irr(G^{m})$ as its character. We can view $\bigotimes^{m} V$ as a $G\wr S_{m}$-module via the action defined by
\[((x_{1},\ \ldots x_{m});\ \pi)\cdot \bigotimes\limits_{i=1}^{m} v_{i} = \bigotimes\limits_{i=1}^{m} x_{i}\cdot v_{\pi^{-1}(i)},\] for all $((x_{1},\ \ldots x_{m});\ \pi)\in G \wr S_{m}$ and $\bigotimes\limits_{i=1}^{m} v_{i} \in \bigotimes^{m} V$. Let $\widetilde{\otimes^{m}\phi}$ denote the character of $G \wr S_{m}$, afforded by $\bigotimes^{m}V$.  It is well known that $\widetilde{\otimes^{m} \phi} \in \irr(G \wr S_{m})$, for all $\phi \in \irr(G)$. We also have $\widetilde{\otimes^{m} \phi}|_{G^{n}}= \otimes^{m} \phi$. As $G^{m} \lhd G \wr S_{m}$, any irreducible character of $S_{m}$ can be identified as an element of $\irr(G \wr S_{m})$. It is well known that the irreducible characters of $S_{m}$ are in one-one correspondence with integer partitions of $m$. Given $\lambda \vdash m$, we denote $\rho_{\lambda}$ to be the irreducible character associated with $\lambda$. 
Given $\lambda \vdash n$ and $\phi \in \irr(G)$, the  character 
\begin{equation}\label{eq:primitive}
\widetilde{\phi^{m}_{\lambda}}:= \widetilde{\otimes ^{m} \phi} \otimes \rho_{\lambda},
\end{equation}
is an irreducible character of $G \wr S_{m}$. Every character in $\irr(G\wr S_{n})$ can be expressed in terms of elements of $\{\widetilde{\phi^{m}_{\lambda}} \ : \ m \leq n\ \&\ \lambda\vdash m\}$.

Given an $\irr(G)$-valued function $f$ on $[n]$, define $\chi_{f}:= \bigotimes\limits_{i=1}^{n} f(i)$. It is well known that given $\chi\in \irr(G^{n})$, there is an $\irr(G)$-valued function $f$ on $[n]$ such that $\chi =\chi_{f}$. Irreducible characters of $G\wr S_{n}$ can be characterized using the well-known method of ``little groups'' due to Wigner and Mackey. Given $g \in G\wr S_{n}$ and $\chi\in \irr(G^{n})$, $\chi^{g}$ is the character satisfying $\chi^{g}(\mbf{x})=\chi(g^{-1}\mbf{x}g)$, for all $\mbf{x}\in G^{n}$. By a {\it conjugate} of $\chi$, we mean a character of the form $\chi^{g}$, for some $g\in G\wr S_{n}$. The {\it inertia subgroup} of $\chi$ is the group $\{g \in G\wr S_{n}\ :\ \chi^{g}=\chi\}$

The set $\{\chi_{f \circ \pi}\ :\ \pi \in S_{n} \}$ is the set of conjugates of $\chi_{f}$. Given $A\subseteq [n]$, let $S_{A} \leq S_{n}$ denote the pointwise stabilizer of $[n]\setminus A$. Given a disjoint partition $\mcal{B}$ of  $[n]$, let $S_{\mcal{B}} \leq S_{n}$ be the subgroup of permutations preserving $\mcal{B}$. A {\it Young subgroup} is a subgroup of the form $S_{\mcal{B}}$, for some partition $\mcal{B}$ of $[n]$.
Let $\mcal{B}_{f}$ be the partition $\{ f^{-1}(\phi)\ :\ \phi\in \irr(G)\}$. Then the inertia subgroup of $\chi_{f}$ in $G \wr S_{n}$, is precisely the group $G \wr S_{\mcal{B}_{f}}$. Consider a partition-valued function $\Lambda$ on $\irr(G)$ such that $\Lambda(\phi)$ is an integer partition of $|f^{-1}(\phi)|$, for all $\phi \in \irr(G)$. Then $ \widetilde{\phi^{|f^{-1}(\phi)|}_{\Lambda(\phi)}} \in \irr(G \wr S_{f^{-1}(\phi)})$ (cf. \eqref{eq:primitive}), and thus \[\bigotimes_{\phi \in \irr(G)} \left(\widetilde{\phi^{|f^{-1}(\phi)|}_{\Lambda(\phi)}} \right) \in \irr(G \wr S_{\mcal{B}_{f}})\] is an irreducible character of the inertia group of $\chi_{f}$, whose restriction to $G^{n}$ has $\chi_{f}$ as an irreducible summand. 

Now, by Clifford theory, given $f:[n] \to \irr(G)$ and given a partition-valued function $\Lambda$ on $\irr(G)$ with $|\Lambda(\phi)|=|f^{-1}(\phi)|$, for all $\phi \in \irr(G)$, we have
\begin{equation}\label{eq:charwreathproduct}
\psi_{f,\Lambda}:= \mrm{Ind}_{G \wr S_{\mcal{B}_{f}}}^{G \wr S_{n}} \left[ \bigotimes_{\phi \in \irr(G)} \left(\widetilde{\phi^{|f^{-1}(\phi)|}_{\Lambda(\phi)}} \right)\right] \in \irr(G\wr S_{n}).
\end{equation}
Observing that we have as many distinct $\psi_{f,\Lambda}$s as the number of conjugacy classes in $G\wr S_{n}$, we note that every irreducible character of $G \wr S_{n}$ is of the form $\psi_{f,\Lambda}$. Before moving on, we note that given $\phi \in \irr(G)$ and $\lambda \vdash n$, we have $\widetilde{\phi^{n}_{\lambda}}= \psi_{f_{\phi}, \Lambda_{\phi,\lambda}}$, where $f_{\phi}:[n] \to \irr(G)$ is such that $f_{\phi}(i)=\phi$, for all $i \in [n]$, and $\Lambda_{\phi,\lambda}$ is the partition-valued function on $\irr(G)$ which takes $\phi$ to $\lambda$ and every other irreducible character to the zero partition.

Taking Theorem~\ref{thm:characterization} into consideration, it helps to compute quantities of the form $\dfrac{\psi(g)-\ov{\psi(g)}}{\psi(1)}$, for $g \in G$ and $\psi \in \irr(G)$. The first such computations we need is the following result.

\begin{lemma}\label{lem:centralcharacterc1}
Consider $f:[n]\to \irr(G)$ and a partition-valued function $\Lambda$ on $\irr(G)$ satisfying $|\Lambda(\phi)|=|f^{-1}(\phi)|$, for all $\phi \in \irr(G)$. Then, we have
\begin{equation*}
\dfrac{\psi_{f,\Lambda}(((x,e,\ldots, e); (1)))-\ov{\psi_{f,\Lambda}(((x,e,\ldots, e); (1)))}}{\psi_{f,\Lambda}(1)} = \sum\limits_{\mu \in \irr(G)}  \dfrac{|f^{-1}(\mu)| (\mu(x)-\ov{\mu(x)})}{n\mu(1)}.
\end{equation*}
\end{lemma}
\begin{proof}
We recall that the set $\{\chi_{f\circ \pi}\ :\ \pi \in S_{n}\}$ is the set of conjugates of $\chi_{f}$.
 Using Clifford theory(\cite[Theorems 6.2 and 6.11]{isaacs1994character}), we have
 \begin{equation}\label{eq:wreathrestriction}
\mrm{Res}^{G \wr S_{n}}_{G^{n}} \psi_{f,\Lambda} = \prod\limits_{\phi \in f\left([n]\right)} \rho_{f, \phi}(1) \times \left[\sum\limits_{h\in\{f\circ \pi\ :\ \pi \in S_{n}\}} \chi_{h}\right].
 \end{equation}
By a conjugate of $f$, we mean an element of the set $\{f\circ \pi\ :\ \pi \in S_{n}\}$. Given $i\in [n]$ and $\mu \in \irr(G)$, let $\Vert f \Vert_{\mu}$ denote the number of conjugates $h$ of $f$, with $h(1)=\mu$. Elementary counting arguments yield that 
\[\Vert f \Vert_{\mu} =  \dfrac{(n-1)!}{\left(|f^{-1}(\mu)|-1\right)! \prod\limits_{\phi \in \irr(G) \setminus \{\mu\}} |f^{-1}(\phi)|! } \]

Using \eqref{eq:wreathrestriction}, we have 
\begin{align*}
\psi_{f,\Lambda}(((x,e,\ldots, e); (1))) & = \prod\limits_{\phi \in \irr(G)} \rho_{\Lambda(\phi)}(1) \times \left[\sum\limits_{\mu \in \irr(G)} \Vert f\Vert_{\mu} \mu (x)  \prod\limits_{\nu \in \irr(G) \setminus \{\mu\}} \nu(1). \right ]
\end{align*}

The result follows by noting that \[\psi_{f,\Lambda}(1) = \left(\prod\limits_{\phi \in \irr(G)} \rho_{\phi, f}(1)\phi(1)\right) \times \Vert f \Vert_ {\mu} \times n\times |f^{-1}(\mu)|^{-1}.\]
\end{proof}

We also make use of the following formula.
\begin{lemma}\label{lem:centralcharacterc2}
Let $f:[n]\to \irr(G)$ be a function, let  $\Lambda$ a partition-valued function on $\irr(G)$ satisfying $|\Lambda(\phi)|=|f^{-1}(\phi)|$, for all $\phi \in \irr(G)$, and let $\pi \in S_{n}$ be an $n$-cycle.
Then the following statements are true for all $\mbf{x}\in G^{n}$.
\begin{enumerate}
    \item If $|f([n])|>1$, then 
    \[\dfrac{\psi_{f,\Lambda}((\mbf{x}; \pi)-\ov{\psi_{f,\Lambda}((\mbf{x}; \pi)}}{\psi_{f,\Lambda}(1)}=0;\] and
    \item if $f([n])=\{\phi\}$, then 
    \[
    |G|^{n-1} \times\dfrac{\psi_{f,\Lambda}((\mbf{x}; \pi)-\ov{\psi_{f,\Lambda}((\mbf{x}; \pi)}}{\psi_{f,\Lambda}(1)} = k_{f,\Lambda}|Z(G)|^{(n-1)} \dfrac{\phi\left( \mbf{x}_{\pi}  \right)-\ov{\phi\left( \mbf{x}_{\pi}  \right)}}{\phi(1) \times (n-1)!},
    \]
    for some $k_{f,\Lambda} \in \bb{Z}$. 
\end{enumerate}
\end{lemma}
\begin{proof}
We recall from the definition \eqref{eq:charwreathproduct} that $\psi_{f,\Lambda}$ is a character induced from the subgroup $G\wr S_{\mcal{B}_{f}}$, where $S_{\mcal{B}_{f}}$ is the subgroup of $S_{n}$ preserving the partition $\mcal{B}_{f}=\{f^{-1}(\mu)\ : \mu \in \irr(G)\}$ of $[n]$. We note that $S_{\mcal{B}_{f}}$ contains an $n$-cycle if and only if $|f([n])|=1$. Therefore, in case $|f([n])|>1$, no element of $G \wr S_{n}$ is conjugate to $(\mbf{x}; \pi)$. Now, the Frobenius formula for induced characters implies part (1).

We now assume that $f([n])=\{\phi\}$. In this case, we have $S_{\mcal{B}_{f}}=S_{n}$; $\lambda:=\Lambda(\phi) \vdash n$ and $\Lambda(\mu)=[0]$, for all $\mu \in \irr(G)\setminus \{\phi\}$; and $\psi_{f,\Lambda}= \widetilde{\phi^{n}_{\lambda}}$. Using the well-known formula (see \cite[Lemma 4.3.9]{James_1984} or \cite[$(8.2)$, Appendix B]{macdonald1998symmetric}) for values of $\widetilde{\phi^{n}_{\lambda}}$, we have 
\begin{align*}
   |G|^{n-1}\times \dfrac{\psi_{f,\Lambda}((\mbf{x}; \pi)-\ov{\psi_{f,\Lambda}((\mbf{x}; \pi)}}{\psi_{f,\Lambda}(1) } & = |G|^{n-1} \times \dfrac{\rho_{\lambda}(\pi)\phi(\mbf{x}_{\pi})-\ov{\rho_{\lambda}(\pi)\phi(\mbf{x}_{\pi})}}{\phi^{n}(1) \rho_{\lambda}(1)}\\
    &= \dfrac{\rho_{\lambda}(\pi)}{\rho_{\lambda}(1)} \times\dfrac{|G|^{n-1}}{\phi^{n}(1)} \dfrac{\phi(\mbf{x}_{\pi})-\ov{\phi(\mbf{x}_{\pi})}}{\phi(1) }. 
\end{align*} 
 It is well known that $\frac{|G|}{\phi(1)} \equiv 0 \pmod{|Z(G)|}$. (see \cite[3.12]{isaacs1994character}). There are exactly $(n-1)!$ number of $n$-cycles in $S_{n}$. Using a well-known result \cite[Theorem 3.7]{isaacs1994character} on integrality of character values and the fact that irreducible characters of $S_{n}$ take on integer values, it follows that $\dfrac{\rho_{\lambda}(\pi)\times (n-1)!}{\rho_{\lambda}(1)}$ is an integer. Thus, the quantity $k_{f,\Lambda}= \dfrac{|G|^{n-1}}{\phi^{n-1}(1) |Z(G)|^{n-1}} \times \dfrac{\rho_{\lambda}(\pi)\times (n-1)!}{{\rho_{\lambda}(1)}}$. This concludes the proof of part (2).  
\end{proof}

\begin{theorem}\label{thm:wreath}
Let $n\geq 2$ be a natural number and let $X:=\mrm{Cay}(G,C)$ be an oriented normal Cayley graph that admits PST from $e$ to $z\in Z(G)$, at a time $\tau$. 
Let 
\begin{subequations}\label{eq:wreathconnex}
\begin{align}
\tilde{C}_{1} &=  \bigcup\limits_{i \in [n]} \left\{((x_{1}, \ldots x_{n}); (1))\ :\ x_{i}\in C\ \& \ x_{j}=e,\ \forall j\neq i\right\}\\
\tilde{C}_{2} &= \left\{(\mbf{x};\pi)\ :\ \text{$\pi$ is an $n$-cycle}\ \&\ \mbf{x}_{\pi}\in C\right\}.
\end{align} 
\end{subequations}
Then $X \wr n := \mrm{Cay}(G\wr S_{n},\ \tilde{C}_{1}\cup \tilde{C}_{2})$ admits PST at time $\tau$ from $(e,e,\ldots, e)$ to $(z,z,\ldots, z)$.
\end{theorem}
\begin{proof}
By Theorem~\ref{thm:characterization}, it suffices to show that 
\[\dfrac{\ov{\psi(z,\ldots, z ) }}{\psi(1)} = \exp\left(\tau \theta_{\psi}\right),\] for all $\psi \in \irr(G \wr S_{n})$. Here,
 \begin{equation}\label{eq:tpsiwr}
\theta_{\psi} = \dfrac{\psi(\tilde{C}_{1})+\psi(\tilde{C}_{2})-\ov{\psi(\tilde{C}_{1})-\psi(\tilde{C}_{2})}}{\psi(1)}.
\end{equation}

Consider some $\psi \in \irr(G \wr S_{n})$.
Following the description of $\irr(G \wr S_{n})$ given at the start of this section, given $\psi \in \irr(G\wr S_{n})$, there exists $f:[n] \to \irr{G}$ and a partition-valued function $\Lambda$ on $\irr(G)$ with $|\Lambda(\phi)|=|f^{-1}(\phi)|$ (for all $\phi \in \irr(G)$) such that $\psi=\psi_{f,\Lambda}$ (as defined in \eqref{eq:charwreathproduct}). We first compute $\theta_{\psi_{f,\Lambda}}$.

\paragraph{Case 1:} Assume that $|f([n])|>1$. Using Lemma~\ref{lem:centralcharacterc2}, it follows that
\[\theta_{\psi_{f, \Lambda}} = \dfrac{\psi_{f, \Lambda}(\tilde{C}_{1}))-\ov{\psi_{f, \Lambda}(\tilde{C}_{1})}}{\psi_{f, \Lambda}(1)}.\] We note that every element of $\tilde{C}_{1}$ is conjugate to an element of the form $((x,e,\ldots ,e) ; (1))$ for some $x\in C$, and that the number of $S_{n}$-conjugates in $\tilde{C}_{1}$ of the element $((x,e,\ldots ,e) ; (1))$ (with $x \in C$) is $n$ . It now follows from Lemma~\ref{lem:centralcharacterc1} that 
\begin{align*}
\dfrac{\psi_{f, \Lambda}(\tilde{C}_{1})}{\psi_{f, \Lambda}(1)} = n\sum\limits_{\phi \in \irr(G)}  \dfrac{|f^{-1}(\phi)|\phi(C)}{n\phi(1)},
\end{align*}
and therefore, if $\theta_{\phi}$ is the eigenvalue associated with $\phi$, of $\mrm{Cay}(G,C)$, we have
\begin{equation}\label{eq:c1sum}
\theta_{\psi_{f, \Lambda}} = \sum\limits_{\phi \in \irr(G)} |f^{-1}(\phi)| \theta_{\phi},\ \text{provided $|f([n]|>1$.}
\end{equation}

 \paragraph{Case 2:} Now, we assume that $f([n])=\{\phi\}$. Set $\lambda:=\Lambda(\phi)$. Fix an $n$-cycle, $\sigma \in S_{n}$. We note that every element in $\tilde{C}_{2}$ is conjugate to an element of the form $(\mbf{x}; \sigma)$ with $\mbf{x}_{\sigma} \in C$. We count that for each $c \in C$, that $X_{c}:=\{(\mbf{x}; \sigma) \ : \mbf{x}_{\sigma} = c \}$ is a set of size $|G|^{n-1}$.  Since every element of $X_{c}$ lies in the same conjugacy class, using Lemma~\ref{lem:centralcharacterc2}, we have
 \begin{align}\label{eq:formulac2}
     \sum\limits_{g \in X_{c}}\dfrac{\psi_{f,\Lambda}(g)-\ov{\psi_{f,\Lambda}(g)}}{\psi_{f, \Lambda}(1)} & =k_{f, \lambda} |Z(G)|^{n-1} \dfrac{\phi(c)-\ov{\phi(c)}}{\phi(1)\times (n-1)!},
 \end{align}
 for some $k_{f,\Lambda}\in \bb{Z}$.
 We recall that there are exactly $(n-1)!$ number of $n$-cycles in $S_{n}$. Thus, we have
\begin{align*}
\dfrac{\psi_{f, \Lambda}(\tilde{C}_{2})-\ov{\psi_{f, \Lambda}(\tilde{C}_{2})}}{\psi_{f, \Lambda}(1)} &= (n-1)!\sum\limits_{c \in C}\sum\limits_{g \in X_{c}}\dfrac{\psi_{f,\Lambda}(g)-\ov{\psi_{f,\Lambda}(g)}}{\psi_{f, \Lambda}(1)} \\
&= \sum\limits_{c \in C}k_{f,\Lambda}|Z(G)|^{n-1} \dfrac{\phi(c) - \ov{\phi(c)}}{\phi(1)} \ \text{(using \eqref{eq:formulac2})}\\
&=k_{f,\Lambda}|Z(G)|^{n-1}\theta_{\phi}.
\end{align*}
Arguing as in Case 1, using Lemma~\ref{lem:centralcharacterc1}, we have
\[\dfrac{\psi_{f, \Lambda}(\tilde{C}_{1})-\ov{\psi_{f, \Lambda}(\tilde{C}_{1})}}{\psi_{f, \Lambda}(1)}=n\theta_{\phi},\] and therefore, we have 

\begin{equation}\label{eq:c2sum}
\theta_{\psi_{f, \Lambda}} = n \theta_{\phi} + k_{f,\Lambda}|Z(G)|^{n-1}\theta_{\phi},\ \text{, for some $k_{f,\Lambda}\in \bb{Z}$, provided $f([n])=\{\phi\}$.}
\end{equation}

Using, \eqref{eq:tpsiwr}, \eqref{eq:c1sum} and \eqref{eq:c2sum}, we have 
\begin{align}\label{eq:taupsi}
\exp\left( \tau \theta_{\psi_{f,\Lambda}}\right) &= \begin{cases} \prod\limits_{\phi \in \irr(G)} \left[ \exp\left( \tau \theta_{\phi} \right) \right]^{|f^{-1}(\phi)|} & \text{if $|f([n])|>1$,}\\
\exp{(\tau\theta_{\phi})}^{n} \times \exp(\tau \theta_{\phi})^{k_{\psi} |Z(G)|^{n-1}} &\text{if $f([n])=\{\phi\}$.} 
\end{cases}
\end{align}
Since $\mrm{Cay}(G, C)$ admits PST from $e$ to $z$ at time $\tau$, by Theorem~\ref{thm:characterization}, we have \[\exp\left( \tau \theta_{\phi} \right) = \frac{\ov{\phi(z)}}{\phi(1)},\] for all $\phi \in \irr(G)$. Applying the above in \eqref{eq:taupsi}, yields

\begin{align*}
\exp\left( \tau \theta_{\psi_{f,\Lambda}}\right) & = \begin{cases}
\prod\limits_{\phi \in \irr(G)} \left( \frac{\ov{\phi(z)}}{\phi(1)} \right)^{|f^{-1}(\phi)|}  & \text{if $|f([n])|>1$,} \\
\left( \dfrac{\ov{\phi(z)}}{\phi(1)} \right)^{n} \times \left(\dfrac{\ov{\phi(z)}}{\phi(1)}\right)^{k_{f,\Lambda}|Z(G)|^{n-1}} & \text{if $f([n])=\{\phi\}$.}
\end{cases}
\end{align*} 
Since $z \in Z(G)$ and since $\dfrac{\ov{\phi(z)}}{\phi(1)}$ is a $|z|$-th root of unity, it follows that $\dfrac{\ov{\phi(z)}}{\phi(1)}^{|Z(G)|}=1$. Therefore, we have
\begin{align*}
\exp\left( \tau \theta_{\psi_{f,\Lambda}}\right) & = \prod\limits_{\phi \in \irr(G)} \left( \dfrac{\ov{\phi(z)}}{\phi(1)} \right)^{|f^{-1}(\phi)|} \\
&= \dfrac{\ov{\psi_{f,\Lambda}(z,\ldots, z)} }{\psi_{f,\Lambda}(1)}.
\end{align*} 
The result now follows from Theorem~\ref{thm:characterization}.
\end{proof}

\begin{remark}
Recall that in \S~\ref{sec:nonsolvable}, we constructed connected oriented normal Cayley graphs on nonsolvable groups, which admit MST on sets of size $3$. It is natural to ask if there are ``nonsolvable'' examples in the case of MST on sets of size $4$. Theorem~\ref{thm:wreath} helps us construct infinitely many such examples. Let $G$, $\mcal{C}$, $\tilde{C}_{1}$, and $\tilde{C}_{2}$ be as in Theorem~\ref{thm:wreath}. We recall the well-known fact that $\mrm{Cay}(G, \mcal{C})$ has as many (strongly) connected components as the index $[G : \left\langle C \right\rangle]$, with each component being isomorphic to $\mrm{Cay}(\left\langle C \right\rangle,\ C)$. Assume that $C$ generates $G$. We note that $\tilde{C}_{1}$ generates $G^{n}$. We recall the fact that $n$-cycles in $S_{n}$ generate $S_{n}$ when $n$ is even and $A_{n}$ when $n$ is odd. Therefore, the subgroup $\left\langle \tilde{C}_{1} \cup \tilde{C_{2}} \right\rangle$ is $G\wr S_{n}$ when $n$ is even and $G\wr A_{n}$, when $n$ is odd. 
As $A_{n}$ is a nonabelian simple group when $n\geq 5$, both $G\wr S_{n}$ and $G \wr A_{n}$ are nonsolvable for $n\geq 5$. Theorem~\ref{thm:wreath} allows us to construct examples of MST of size 4 in connected oriented normal Cayley graphs of nonsolvable groups. For example, if $X$ is one of the examples in \cite[10.2]{GodsilZhan} which exhibits MST of size 4, then $X\wr 6$, is a connected oriented normal Cayley graph on a group with $S_{6}$ as one of its composition factors. Theorem~\ref{thm:wreath} asserts that $X \wr 6$ admits MST on a set of size $4$. 
\end{remark}  

We end this section by observing that a version of Theorem~\ref{thm:wreath} for unoriented normal Cayley graphs holds {\it mutatis mutandis}. We state the same, without proof.

\begin{theorem}\label{thm:wreathundirected}
Let $n\geq 2$ be a natural number and let $X:=\mrm{Cay}(G,C)$ be a normal Cayley graph that admits PST from $e$ to $z\in Z(G)$, at a time $\tau$. 
Let 
\begin{subequations}
\begin{align*}
\tilde{C}_{1} &=  \bigcup\limits_{i \in [n]} \left\{((x_{1}, \ldots x_{n}); (1))\ :\ x_{i}\in C\ \& \ x_{j}=e,\ \forall j\neq i\right\}\\
\tilde{C}_{2} &= \left\{(\mbf{x};\pi)\ :\ \text{$\pi$ is an $n$-cycle}\ \&\ \mbf{x}_{\pi}\in C\right\}.
\end{align*} 
\end{subequations}
Then $X \wr n := \mrm{Cay}(G\wr S_{n},\ \tilde{C}_{1}\cup \tilde{C}_{2})$ admits PST at time $\tau$ from $(e,e,\ldots, e)$ to $(z,z,\ldots, z)$.    
\end{theorem}


\end{document}